\documentclass[12pt,a4paper]{article}
\usepackage[utf8]{inputenc}
\usepackage{amsmath}
\usepackage{amsfonts}
\usepackage{amssymb}
\usepackage{amsthm}
\usepackage{mathtools}
\usepackage{graphicx}
\usepackage{hyperref}

\theoremstyle{definition}
\newtheorem{definition}{Definition}
\newtheorem{proposition}[definition]{Proposition}

\theoremstyle{remark}
\newtheorem{remark}[definition]{Remark}

\newcommand{\R}{\mathbb{R}}

\newcommand{\Z}{\mathbb{Z}}
\newcommand{\C}{\mathbb{C}}

\title{Geometric phase and holonomy in the space of 2-by-2 symmetric operators}
\author{Jakub Rondomanski, José D. Cojal González, Jürgen P. Rabe,\\ Carlos-Andres Palma, Konrad Polthier 
\\ \\
\small \href{jakub.rondomanski@fu-berlin.de}{jakub.rondomanski@fu-berlin.de}, \href{konrad.polthier@fu-berlin.de}{konrad.polthier@fu-berlin.de}, 
\href{palma@iphy.ac.cn}{palma@iphy.ac.cn}
}

\date{December 18, 2023}
\begin{document}
\bibliographystyle{alpha}
\maketitle
\begin{abstract}
We present a non-trivial metric tensor field on the space of 2-by-2 real-valued, symmetric matrices whose Levi-Civita connection renders frames of eigenvectors parallel. This results in fundamental reimagining of the space of symmetric matrices as a curved manifold (rather than a flat vector space) and reduces the computation of eigenvectors of one-parameter-families of matrices to a single computation of eigenvectors at an initial point, while the rest are obtained by the parallel transport ODE. Our work has important applications to vibrations of physical systems whose topology is directly explained by the non-trivial holonomy of the spaces of symmetric matrices.
\end{abstract}
\tableofcontents
\section{Introduction}
The space of real-valued, symmetric $n$-by-$n$ matrices is a vector subspace of the space of all real-valued $n$-by-$n$ matrices, which is itself isomorphic to the Euclidean space $\R^{n^2}$. The symmetric matrices thus inherit a Euclidean topology. It is also customary to consider them geometrically trivial: as a vector space, one typically associates a norm to it, be it the Frobenius norm, the operator norm, or any other. Indeed, the operator norm is often considered a “natural norm” of matrix spaces. We propose a different natural geometric structure specifically for the spaces of symmetric matrices by focusing on a fundamental property: the existence of orthogonal frames of eigenvectors. We construct a metric tensor on the tangent bundle of the space of symmetric matrices so that the frame fields consisting of unit-length eigenvectors become parallel with respect to the associated Levi-Civita connection. This metric is not completely flat: the singular subspaces consisting of matrices with repeating eigenvalues obtain infinite curvature and lead to non-trivial holonomy. Thus, we propose a reimagination of the space of symmetric matrices from a globally flat vector space to a curved, conical space, based on the natural variation of eigenvectors. 

The non-trivial holonomy impacts the behaviour of various physical systems. Fields such as geometry processing, molecular dynamics, or solid-state physics, can have large parts of their data represented by a single symmetric matrix. In mechanical systems it describes the second order dynamics of the system around an equilibrium. The eigenstructure (eigenvectors and eigenvalues) describes the directions and the frequencies of resonant vibrations. However, the dynamical matrix is also of interest in the study of statical models, as the eigenstructure contains information about the underlying geometry.

In total, the dynamical matrix depends on three structures of the underlying model: the discrete/combinatorial structure which describes how the points, edges and potentially higher-dimensional cells are connected, the geometric structure which describes the geometric position of the cells in space, and the physical structure which consists of material properties, for example the masses of points and the spring constants of the edges. It is central to the properties of the eigenstructure of a system that it depends on all three of the structures.

In most real-life scenarios, the system and its dynamical matrix are not constant but depend on one or multiple parameters. In particular the physical properties should be considered as functions of time, in general. It is therefore essential to be able to compute eigenstructures of parameter-dependent matrices. We present a general, mathematically rigorous framework for large-scale variations of eigenvectors, which complements results from perturbation theory dealing with infinitesimal variations. Our work is a steppingstone in the tabulation of Riemannian connections and metrics of $n$-by-$n$ matrices with important implications in eigenproblems and topological physics.

The work is inspired by insights in various physical phenomena in which a direction depends on the configuration parameters, but is not its function, meaning that the direction can be different even if the underlying parameters are equal. A classic example is Foucault’s pendulum. It was Berry \cite{Berry} who collected such phenomena and interpreted them as an “anholonomy” due to parallel transport with respect to a non-trivial connection on an abstract underlying space of parameters. His work inspired entire fields of physics, in particular ones under the umbrella term “topological physics”, which notably won the 2016 Nobel Prize in Physics and has inspired a wealth of new research. 

We complement this body of research by providing explicit, analytic formulae for the Berry connection in the case of 2-by-2 real-valued, symmetric matrices and by reinterpreting it as a Levi-Civita connection of a metric tensor, for which we also provide explicit formulae.
Hence, in our approach, the Berry connection is intrinsically defined from the metric and does not depend on an embedding in a flat space. The standard formulae for the Berry connection $\omega$ and Berry curvature $d\omega$ of a 2-dimensional manifold $M \subset \R^3 \subset \C^3$ embedded in the flat space $\R^3$ with respect to a local orthonormal frame $(e_1, e_2)$ on $TM$ are:
\begin{align}
	\omega_i &= \langle e_1, \partial_i e_2 \rangle \\
	(d\omega)_{ij} &= \langle \partial_i e_1, \partial_j e_2 \rangle - \langle \partial_j e_1, \partial_i e_2 \rangle
\end{align}
or equivalently in the complexified form for $n \coloneqq \frac{1}{\sqrt{2}}\left(e_1 +ie_2\right)$:
\begin{align}
	\omega_i &= Im\langle n, \partial_i n \rangle \\
	(d\omega)_{ij} &= Im(\langle \partial_i n, \partial_j n \rangle - \langle \partial_j n, \partial_i n \rangle)
\end{align}
where $\langle .,. \rangle$ is the standard Euclidean/Hermitian product on $\R^3$ or $\C^3$, respectively. The differential $\partial_i e_2$ corresponds to the flat connection on $\R^3$; in particular, it has a normal component:
\begin{align}
	\partial_i e_2 = \langle \partial_i e_2, e_1 \rangle e_1 + II(e_1, \partial_i) \nu
\end{align}
where $\nu$ is the unit normal on $M$ and $II$ is its second fundamental form. 
This normal component vanishes in (1) but not in (2).
We propose instead an intrinsic formulation:
\begin{align}
	\omega_i &= g( e_1, \nabla_i e_2 ) \\
	(d\omega)_{ij} &= -R(\partial_i, \partial_j, e_1, e_2) = -g(\nabla_i\nabla_je_1-\nabla_j\nabla_ie_1, e_2)
\end{align}
where $g$ is the pullback-metric due to the embedding $M \subset \R^3$, $\nabla$ its Levi-Civita connection and $R$ the (purely covariant) Riemann curvature tensor.
See Appendix for a derivation of both formulations and their equivalence.
\section{Symmetric matrices with 0 trace}
\begin{definition}
Let 
\begin{align}
	\Sigma \coloneqq \left\{ \begin{pmatrix}
		x & y \\ y & -x
	\end{pmatrix}
	: x,y \in \R \right\} \cong \R^2
\end{align}
be the space of 2-by-2 real-valued symmetric matrices with no trace.
This is a 2-dimensional vector space and a smooth manifold with a global cartesian coordinate chart
\begin{align}
	\Phi : \Sigma &\to \R^2 \\
	\begin{pmatrix}
		x & y \\ y & -x
	\end{pmatrix}
	&\mapsto (x, y)
\end{align}

We will also use the polar coordinates $(r, \phi)$ given by 
\begin{align}
	(x, y) = (r\cos(\phi), r\sin(\phi)).
\end{align}
The partial derivatives of cartesian coordinates are the standard unit vector fields
\begin{align}
	\partial_x = \partial_1 = \begin{pmatrix} 1\\0 \end{pmatrix}, \qquad \partial_y = \partial_2 = \begin{pmatrix} 0\\1 \end{pmatrix}.
\end{align}
The partial derivatives of polar coordinates are 
\begin{align}
	\partial_r &= \cos(\phi)\partial_x + \sin(\phi)\partial_y \\
	\partial_\phi &= -r\sin(\phi)\partial_x + r\cos(\phi) \partial_y.
\end{align}
\end{definition}
\begin{definition}[Metric]
	Let\begin{align}
		f : \Sigma &\to \R^3 \\
		(x,y) &\mapsto \left(x, y, \sqrt{3} \sqrt{x^2+y^2}\right) = \left(x, y, \sqrt{3}r\right) 
	\end{align}
be the parametrization of a cone.
It is smooth everywhere apart from the origin.
Its Jacobian is:
\begin{align}
	Df = \begin{pmatrix}
		1&0\\
		0&1 \\
		\sqrt{3}x/r & \sqrt{3}y/r
	\end{pmatrix}
\end{align}
and the resulting metric is:
\begin{align}
	g = Df^\top Df = \begin{pmatrix}
		1+3x^2/r^2 & 3xy/r^2 \\
		3xy/r^2 & 1+3y^2/r^2
	\end{pmatrix}
\end{align}
The matrix $g$ contains the scalar products $g_{ij} = g(\partial_i, \partial_j)$.
Given a smooth curve $\gamma : [0, 1] \to \Sigma\setminus \{0\}$ with velocity field $\gamma'(t) = \gamma'^1(t) \partial_1 + \gamma'^2(t) \partial_2$, the length of $\gamma$ (with respect to $g$) is:
\begin{align}
	L(\gamma) = \int_0^1 \sqrt{g_{ij}\gamma'^i \gamma'^j} dt
\end{align}
\end{definition}
\begin{definition}[Frames]
	We seek a reference orthonormal frame to define the connection and curvature forms.
	The cartesian coordinate frame $(\partial_1, \partial_2)$ is not orthonormal in the metric $g$. 
	Indeed, $g(\partial_i, \partial_j) = g_{ij} \neq \delta_{ij}$.
	However, the polar coordinate vectors are orthogonal:
	\begin{align}
		g(\partial_r, \partial_\phi) &= 0 \\
		g(\partial_r, \partial_r) &= 4 \\
		g(\partial_\phi, \partial_\phi) &= r^2.
	\end{align}
	In other words, the metric in polar coordinates is diagonal:
	\begin{align} \label{polar-metric}
		g_{pol} = 
		\begin{pmatrix}
			g(\partial_r, \partial_r) & g(\partial_r, \partial_\phi) \\
			g(\partial_r, \partial_\phi) & g(\partial_\phi, \partial_\phi)
		\end{pmatrix}
		=
		\begin{pmatrix}
			4 & 0 \\
			0 & r^2
		\end{pmatrix}
	\end{align}
	Thus, the normalized polar coordinate vectors $\left( \frac{\partial_r}{2}, \frac{\partial_\phi}{r}\right)$ are an orthonormal frame on $\Sigma\setminus \{0\}$.	
	Ideally, the reference frame would be as ``constant'' as possible. 
	The polar coordinate frame however is rotating around the origin.
	We can rotate it backwards and thus define the reference orthonormal frame $(e_1, e_2)$ as:
	\begin{align}
		e_1 &\coloneqq \cos(\phi)\frac{\partial_r}{2} - \sin(\phi)\frac{\partial_\phi}{r} \\
			&= \left( \frac{\cos^2(\phi)}{2}+\sin^2(\phi)\right)\partial_x - \frac{\cos(\phi)\sin(\phi)}{2}\partial_y  \\
		e_2 &\coloneqq \sin(\phi)\frac{\partial_r}{2} + \cos(\phi)\frac{\partial_\phi}{r} \\
			& = -\frac{\cos(\phi)\sin(\phi)}{2}\partial_x + \left(\frac{\sin^2(\phi)}{2}+\cos^2(\phi)\right)\partial_y
	\end{align}
\end{definition}
\begin{definition}[Connection, curvature]
	Let $\nabla$ be the Levi-Civita connection of $g$ and $R$ its curvature tensor.
	Using the orthonormal frame $(e_1, e_2)$ we can fully describe them by the connection 1-form
	\begin{align}
		\omega : T\Sigma &\to \R \\
		X &\mapsto \omega(X) \coloneqq g(e_1, \nabla_X e_2) = -g(\nabla_X e_1, e_2)
	\end{align}
	and the curvature 2-form
	\begin{align}
		d\omega : T\Sigma \times T\Sigma &\to \R \\
		(X,Y) &\mapsto d\omega(X, Y) = R(X, Y, e_2, e_1)
	\end{align}
(see Appendix \ref{app:conn} for more details).
\end{definition}
\begin{proposition}
	Let $\omega$, $d\omega$ be the connection and curvature forms of $g$ with respect to the orthonormal frame $(e_1, e_2)$. Then:
	\begin{enumerate}
		\item $\omega = \frac{1}{2r^2} g(\partial_\phi, \_) = \frac{1}{2}d\phi $ 
		\item $d\omega = \pi \delta(x,y) dA$ 
	\end{enumerate}
where $\delta(x,y)$ is the 2-dimensional delta function.
\end{proposition}
\begin{proof}
	\begin{enumerate}
		\item In a 2-dimensional space it is
		\begin{align}
			\omega(X) = g(e_1, \nabla_X e_2) = g([e_1, e_2], X)
		\end{align}
	where $[.,.]$ is the Lie-bracket of vector fields:
	\begin{align}
		[e_1, e_2] &= \partial_{e_1}e_2 - \partial_{e_2}e_1 \\
		&= \left(-\frac{\sin(\phi)\cos(\phi)}{2r} \partial_r  +\frac{2 \sin^2(\phi)-\cos^2(\phi)}{2r^2} \partial_\phi \right) \\
		&\quad - \left(-\frac{\sin(\phi)\cos(\phi)}{2r} \partial_r  +\frac{\sin^2(\phi)-2\cos^2(\phi)}{2r^2} \partial_\phi \right) \\
		&=\frac{1}{2r^2}\partial_\phi
	\end{align}
	Let $\omega_r, \omega_\phi: \Sigma \setminus \{0\} \to \R$ be component functions so that
	\begin{align}
		\omega = \omega_r dr + \omega_\phi d\phi
	\end{align}
	Then:
	\begin{align}
		\omega_r &= \omega(\partial_r) = 0 \\
		\omega_\phi &= \omega(\partial_\phi) = \frac{g(\partial_\phi, \partial_\phi)}{2r^2} = \frac{1}{2}
	\end{align}
	\item For $(x, y) \neq (0,0)$ we can compute in polar coordinates:
	\begin{align}
		\nabla_re_1 &= -\omega(\partial_r)e_2 = 0 \\
		\nabla_\phi e_1 &= -\omega(\partial_\phi)e_2 = -\frac{1}{2}e_2 \\
		\nabla_r \nabla_\phi e_1 &= -\frac{1}{2} \nabla_r e_2 = -\frac{1}{2} \omega(\partial_r) e_1 = 0 \\
		\nabla_\phi \nabla_r e_1 &= \nabla_\phi 0 = 0
	\end{align}
	from which it already follows that $R=0$.
	We can test for the value at $0$ by integrating over arbitrarily small circles.
	Let $r>0$ be any radius and define 
	\begin{align}
		\gamma_r : [0, 2\pi] &\to \Sigma \\
		t &\mapsto r(\cos(t), \sin(t)).
	\end{align}
	Let $U_r$ be the disc around 0 with radius $r$, i.e. $\partial U_r = im(\gamma_r)$.
	It follows by Stokes:
	\begin{align}
		\int_{U_r} d\omega &= \int_{im(\gamma_r)} \omega \\
		&= \int_0^{2\pi} \omega(\gamma_r'(t)) dt \\
		&= \int_0^{2\pi} \frac{1}{2r^2}g(\partial_\phi, -r\sin(t) \partial_x + r\cos(t) \partial_y) dt \\
		&= \int_0^{2\pi} \frac{1}{2r^2}g(\partial_\phi, \partial_\phi) dt \\
		&= \int_0^{2\pi} \frac{1}{2}dt \\
		&= \pi
	\end{align}
	which equals the value of $\int_{U_r} \pi \delta(x,y) dA$.
	\end{enumerate}
\end{proof}
\begin{proposition} \label{prop:berry}
	The connection $\nabla$ is the Berry connection in the following sense: Let 
	\begin{align}
		E: U \subset &\Sigma \to T\Sigma \\
		(x,y) &\mapsto E(x,y) = E^1(x,y) e_1 + E^2(x,y) e_2
	\end{align}
	be a local smooth vector field with component functions $E_1, E_2 : U \to \R$ so that for each $(x,y) \in U$ the vector $(E^1(x,y), E^2(x,y))^\top$ is a unit-length eigenvector of the matrix $\begin{pmatrix} x & y \\ y & -x \end{pmatrix}$.
	Then $\nabla E \equiv 0$.
\end{proposition}
\begin{proof}
	There are 4 unit-length eigenvectors at each point.
	It is enough to show the proposition for one smooth choice of $E$.
	The other choices, namely $-E$ and rotations by $\pm \pi/2$, follow automatically.
	Set 
	\begin{align}
		E \coloneqq \cos\left(\frac{\phi}{2}\right)e_1 + \sin\left(\frac{\phi}{2}\right)e_2
	\end{align}
	(see Appendix \ref{app:eig} for proof that it is an eigenvector).
	Then it is:
	\begin{align}
		\nabla_rE &= \cos(\frac{\phi}{2})\nabla_r e_1 + \sin(\frac{\phi}{2})\nabla_r e_2 \\
		&= -\cos(\frac{\phi}{2})\omega(\partial_r) e_2 + \sin(\frac{\phi}{2})\omega(\partial_r) e_1 \\
		&= 0 \\
		\nabla_\phi E &=  d_\phi (\cos(\frac{\phi}{2}))e_1 + \cos(\frac{\phi}{2}) \nabla_\phi e_1 +  d_\phi(\sin(\frac{\phi}{2}))e_2 + \sin(\frac{\phi}{2}) \nabla_\phi e_2 \\
		&= -\frac{1}{2} \sin(\frac{\phi}{2}) e_1 -\frac{1}{2} \cos(\frac{\phi}{2}) e_2 + \frac{1}{2} \cos(\frac{\phi}{2}) e_2 + \frac{1}{2} \sin(\frac{\phi}{2}) e_1 \\
		&= 0
	\end{align}
\end{proof}
It follows that the forms $\omega, d\omega$ represent the Berry connection and Berry curvature.
In particular, we can now define the geometric phase for any (not necessarily closed) curve.
\begin{definition}[Geometric phase]
	Let $\gamma : [0, 1] \to \Sigma$ be a smooth curve. 
	The geometric phase $\theta(\gamma)$ along $\gamma$ is the rotation angle due to parallel transport when measured in the reference frame $(e_1, e_2)$.
	It is given by:
	\begin{align}
		\theta(\gamma) \coloneqq \int_{im(\gamma)}\omega = \int_0^1 \omega(\gamma'(t))dt = \int_0^1 \frac{1}{2r^2}g(\partial_\phi, \gamma'(t)) dt
	\end{align}
\end{definition}
Note that while the parallely transported vector is itself well-defined, to measure an angle one needs to choose a reference frame.
An exception is when the curve is closed: $\gamma(0) = \gamma(1)$.
Then, the geometric phase is determined solely by the winding number of the curve.
\begin{definition}[Winding number]
	Let $\gamma : [0, 1] \to \Sigma$ be a smooth curve.
	Its winding number (around 0) is the following integral:
	\begin{align}
		W(\gamma; 0) = \frac{1}{2\pi}\int_{im(\gamma)}d\phi
	\end{align}
	If $\gamma$ never goes through 0, then its winding number is an integer.
	However, the winding number is well-defined also for smooth curves which do cross the origin.
	For example, a simple closed curve which crosses 0 once has the winding number of $1/2$.
	See $\cite{hungerbuehler}$ for a full discussion on non-integer winding numbers.	
\end{definition}
\begin{proposition}[Geometric phase of closed curves]
	Let $\gamma : [0, 1] \to \Sigma$ be a smooth curve with winding number $W(\gamma; 0)$.
	The geometric phase $\theta(\gamma)$ along $\gamma$ is:
	\begin{align}
		\theta(\gamma) = \pi \cdot W(\gamma; 0)
	\end{align}
\end{proposition}
\begin{proof}
	\begin{align}
		\theta(\gamma) = \int_{im(\gamma)}\omega = \frac{1}{2}\int_{im(\gamma)}d\phi = \pi \cdot W(\gamma; 0)
	\end{align}
\end{proof}
\begin{proposition}[Holonomy] \label{prop:holonomy}
	The holonomy groups of $\Sigma \setminus \{0\}$, respectively of $\Sigma$ are given by:
	\begin{enumerate}
		\item $Hol(\Sigma \setminus \{0\}) = \Z_2$
		\item $Hol(\Sigma) = \Z_4$
	\end{enumerate}
\end{proposition}
\begin{proof}
	\begin{enumerate}
		\item A generating element of the fundamental group of $\Sigma \setminus \{0\}$ is a simple closed loop with winding number 1.
		It has geometric phase of $\pi$. 
		It follows for the holonomy group:
		\begin{align}
			Hol(\Sigma \setminus \{0\}) = \{(k \cdot \pi) \text{  mod  } 2\pi : k \in \Z \} = \{0, \pi\} \cong \Z_2
		\end{align}
		\item The smallest geometric phase in a closed loop is given by a simple closed loop through 0, i.e. a closed loop with winding number $1/2$.
		It follows for the holonomy group:
		\begin{align}
			Hol(\Sigma) = \left\{\left(k \cdot \frac{\pi}{2}\right) \text{  mod  } 2\pi : k \in \Z \right\} = \left\{0, \frac{\pi}{2}, \pi\, \frac{3\pi}{2} \right\} \cong \Z_4
		\end{align}
	\end{enumerate}
\end{proof}

\section{General symmetric matrices}
We now extend our definitions to 2-by-2 symmetric matrices with any trace, i.e. to the space
\begin{align}
	Sym(2,\R) \coloneqq \left\{ A \in \R^{2\times 2} : A^\top = A \right\} \cong \R^3
\end{align}
with a global cartesian coordinates chart 
\begin{align} \label{cart-coords}
	\Phi : Sym(2,\R) &\to \R^3 \\
	\begin{pmatrix}
		x +z & y \\ y & -x +z
	\end{pmatrix}
	&\mapsto (x, y, z)
\end{align}
The space $\Sigma$ from previous chapter is a plane given by the points $(x,y,0)$.

The extension is straight-forward, as the directions of eigenvectors do not depend on the trace of the matrix.
We can thus simply take a direct sum of $\Sigma$ with a straight line representing the $z$-coordinate.
This leads to the following metric.
\begin{definition}[Metric]
	Let the space $Sym(2,\R)$ carry the metric:
	\begin{align}
		g = \begin{pmatrix}
			1+3x^2/r^2 & 3xy/r^2 & 0 \\
			3xy/r^2 & 1+3y^2/r^2 & 0 \\
			0 & 0 & 1
		\end{pmatrix}.
	\end{align}
\end{definition}
Let $(r, \phi, z)$ be the cylindrical coordinates of $(x,y,z)$ (i.e., $(r, \phi)$ are the polar coordinates of $(x,y)$, as before).
They are well-defined everywhere apart from the singular line:
\begin{align}
	L \coloneqq \{(0, 0, z) \in Sym(2, \R) : z \in \R\}
\end{align}
\begin{definition}[Frame]
	We can extend the orthonormal reference frame from previous chapter to $(e_1, e_2, e_3)$ given by:
	\begin{align}
		e_1 &\coloneqq \cos(\phi)\frac{\partial_r}{2} - \sin(\phi)\frac{\partial_\phi}{r} \\
		e_2 &\coloneqq \sin(\phi)\frac{\partial_r}{2} + \cos(\phi)\frac{\partial_\phi}{r} \\
		e_3 &\coloneqq \partial_z
	\end{align}
\end{definition}
\begin{definition}[Connection, curvature]
	The connection form now has, technically, 3  independent components: 
	\begin{align}
		\nabla \begin{pmatrix}e_1 \\ e_2 \\ e_3 \end{pmatrix} =
		\begin{pmatrix}
			0 & -\omega_{12} & -\omega_{13} \\
			\omega_{12} & 0 & -\omega_{23} \\
			\omega_{13} & \omega_{23} & 0
		\end{pmatrix}
	\begin{pmatrix}e_1 \\ e_2 \\ e_3 \end{pmatrix}
	\end{align}
	However, we have here
	\begin{align}
		\omega_{13} \equiv \omega_{23} \equiv 0
	\end{align}
	and so we may use the same notation as before:
	\begin{align}
		\omega(X) \coloneqq \omega_{12}(X) = g(e_1, \nabla_X e_2) = \frac{1}{2r^2}g(\partial_\phi, X) = \frac{1}{2}d\phi.
	\end{align}
	For the curvature form we obtain again:
	\begin{align}
		d\omega = \pi \delta(x,y) dA
	\end{align}
	where the area element is given by $dA = e_1^\flat \wedge e_2^\flat$ and the 2D-delta function is equal to 0 everywhere apart from the singular line $L$.
	It is understood so that if a 2-dimensional domain $U \subset Sym(2, \R)$ crosses the line $L$ $k$-times, then:
	\begin{align}
		\int_U d\omega = k \cdot \pi.
	\end{align}
\end{definition}
\begin{definition}[Geometric phase]
	Let $\gamma : [0, 1] \to Sym(2, \R)$ be a smooth curve. 
	The geometric phase $\theta(\gamma)$ along $\gamma$ as measured in the reference frame $(e_1, e_2, e_3)$ is given by:
	\begin{align}
		\theta(\gamma) \coloneqq \int_{im(\gamma)}\omega = \int_0^1 \omega(\gamma'(t))dt = \int_0^1 \frac{1}{2r^2}g(\partial_\phi, \gamma'(t)) dt
	\end{align}
	When the curve is closed: $\gamma(0) = \gamma(1)$, we have:
	\begin{align}
		\theta(\gamma) = \pi \cdot W(\gamma; L)
	\end{align}
	where $W(\gamma; L)$ is the winding number of $\gamma$ around the singular line $L$.
\end{definition}
\begin{proposition}
	The connection $\nabla$ is the Berry connection in the following sense: Let 
	\begin{align}
		E: U \subset &Sym(2, \R) \to T Sym(2, \R)\\
		(x,y,z) &\mapsto E(x,y,z) = E^1(x,y,z) e_1 + E^2(x,y,z) e_2
	\end{align}
	be a local smooth vector field with component functions $E^1, E^2 : U \to \R$ so that for each $(x,y,z) \in U$ the vector $(E^1(x,y,z), E^2(x,y,z))^\top$ is a unit-length eigenvector of the matrix $\begin{pmatrix} x+z & y \\ y & -x+z \end{pmatrix}$.
	Then $\nabla E \equiv 0$.
\end{proposition}
\begin{proof}
	The vector field $E \coloneqq \cos\left(\frac{\phi}{2}\right)e_1 + \sin\left(\frac{\phi}{2}\right)e_2$ is a field of unit-length eigenvectors (see Appendix \ref{app:eig}).
	It follows by the same computation as in Proposition \ref{prop:berry} that $\nabla E \equiv 0$.
\end{proof}
\begin{proposition}[Holonomy]
	The holonomy groups of $Sym(2,\R) \setminus L$, respectively of $Sym(2,\R)$ are given by:
	\begin{enumerate}
		\item $Hol(Sym(2,\R) \setminus L) = \Z_2$
		\item $Hol(Sym(2,\R)) = \Z_4$
	\end{enumerate}
\end{proposition}
\begin{proof}
	The proof is the same as the proof of Proposition \ref{prop:holonomy}, except that the generators are now loops with winding numbers of 1, respecitvely $1/2$, being measured around the line $L$ instead of the point 0.
\end{proof}
We have so far shown that local fields of eigenvectors $E$ are parallel, i.e. they satisfy the equation $\nabla E \equiv 0$.
Conversely, we can find eigenvectors along smooth curves given an eigenvector at an initial point by solving the parallel transport ODE:
\begin{proposition}
	Let $\gamma : I = [0, 1] \to Sym(2,\R)$ be a smooth curve and let $E_0 \in T_{\gamma(0)}Sym(2,\R)$ be an eigenvector of $\gamma(0) \in Sym(2,\R)$ (in the sense that $E_0 = E_0^1 e_1(\gamma(0)) + E_0^2 e_2(\gamma(0))$ and $(E_0^1, E_0^2)^\top$ is an eigenvector of $\gamma(0)$).
	Furthermore, let $E : I \to T(im(\gamma)) \subset TSym(2,\R)$ be a vector field along $\gamma$ which satisfies the initial value problem:
	\begin{align}
		E(0) &= E_0 \\
		\forall t \in I : \nabla_{\gamma'(t)} E(t) &= 0.
	\end{align}
	Then $E(t)$ is an eigenvector of $\gamma(t)$ for each $t \in I$.
\end{proposition}

\begin{remark}[Summary/Open problems]
	The above proposition provides a numerical method for computing eigenvectors along a smooth 1-parameter-family of symmetric matrices, or a discrete sampling thereof.
	If the sample size $N \gg 0$ is large, instead of solving the algebraic eigenvector problem $N$ times, with our method one only needs to do it once and then use a numerical scheme for the parallel transport ODE.
	Of course, for 2-by-2 matrices, one does need to worry about computation speeds. 
	However, for large matrices, the algebraic eigenvector computations become very costly.
	The problem lies in computing the appropriate metric and the connection coefficients beforehand.
	Further research is needed to provide an extension of the conic metric to higher-dimensional matrix spaces.
\end{remark}

\section{Unwinding in a Covering Space}
In the previous sections, we described the rotation of eigenvectors when moving through the space of symmetric matrices by means of a metric tensor field.
This was a geometric, local approach.
In this section, we will concentrate on the main inconvenience, which is that closed curves winding once around the line $L$ do not map eigenvectors back to themselves but to their negatives (in other words: the goemetric phase along such curves is $\pi$).
We will use global, topological methods to deal with it.

The main idea is to track the eigenvectors when moving in the space of matrices and use those trackers to differentiate whether the number of windings of a closed curve was odd or even.
In other words, given a closed curve, we want to consider a given starting point matrix $A_0 \in Sym(2,\R)$ as being \textit{different} to the endpoint matrix $A_1$, if the eigenvectos have flipped when going from $A_0$ to $A_1$, even when the underlying matrices are equal.
The way to achieve this tracking is to lift the curves from $Sym(2,\R)$ to a covering space.

It should be enough for the tracker to have a binary true-false value: either the vectors have flipped or they have not.
In other words, we are looking for a \textit{double covering} of $Sym(2,\R)$.
A natural way to do it is to ``unwind" the space $Sym(2,\R)$ by a rotation around the singular line $L$.
As the eigenvectors do not change along $L$, we will start with the 2-dimensional plane $\Sigma$ for better clarity.
\begin{definition}
	Let $\bar{\Sigma} \coloneqq \R^2 \cong \C$.
	Consider the mapping:
	\begin{align}
		\pi_{\Sigma} : \bar{\Sigma} &\to \Sigma \\
		re^{i\phi} &\mapsto re^{i2\phi}
	\end{align}
	We call it the \textit{natural double covering of $\Sigma$ with branching point 0}
	and we call the space $\bar{\Sigma}$ the \textit{branched covering space} of $\Sigma$.
\end{definition}

\begin{remark}
	For each non-zero $A = re^{i\phi} \in \Sigma$ the preimage contains two points:
	\begin{align}
		\pi_{\Sigma}^{-1}(A) &= \left\lbrace re^{i\frac{\phi}{2}}, re^{i \left( \frac{\phi}{2}+\pi \right)} \right\rbrace \\
		&= \{ \pm re^{i\frac{\phi}{2}} \}
	\end{align}
	while for $A=0$ it has only one point:
	\begin{align}
		\pi_{\Sigma}^{-1}(0) = \{0\}.
	\end{align}
	Hence the origin is considered a branching point.
\end{remark}

\begin{proposition}[Geometry of covering space]
	\begin{enumerate}
		\item The pullback metric $h \coloneqq \pi_{\Sigma}^*g$ is given in polar coordinates by:
		\begin{align}
			h_{pol} = \begin{pmatrix}
				4 & 0 \\ 0 & 4r^2
			\end{pmatrix}
		\end{align}
		\item The connection form $\omega^h$ for the Levi-Civita connection of $h$ is:
		\begin{align}
			\omega_h = d\phi
		\end{align}
		\item The geometric phase along a closed smooth curve $\gamma : S^1 \to \bar{\Sigma}$ is:
		\begin{align}
			\theta(\gamma) = 2\pi \cdot W(\gamma; 0)
		\end{align}
		\item The holonomy group of $\bar{\Sigma} \setminus \{0\}$, respectively of $\bar{\Sigma}$, is:
		\begin{align}
			Hol(\Sigma \setminus \{0\}) &= 0 \\
			Hol(\Sigma) &= \Z_2
		\end{align}
	\end{enumerate}
\end{proposition}
\begin{proof}
	\begin{enumerate}
		\item The Jacobian $D\pi_{\Sigma}$ maps $\partial_r$ to $\partial_r$ and $\partial_\phi$ to $2\partial_\phi$.
		It follows:
		\begin{align}
			h(\partial_r, \partial_r) &= g(\partial_r, \partial_r) = 4 \\
			h(\partial_\phi, \partial_\phi) &= g(2\partial_\phi, 2\partial_\phi) = 4(\partial_\phi, \partial_\phi) = 4r^2 \\
			h(\partial_r, \partial_\phi) &= g(\partial_r, 2\partial_\phi) = 0
		\end{align}
		\item Let $\nabla^h$ be the Levi-Civita connection of $h$ and let $(f_1, f_2)$ be the pullback frame of $(e_1, e_2)$.
		Then the connection form $\omega^h$ of $\nabla^h$ with respect to the frame $(f_1, f_2)$ is the pullback of the connection form $\omega$.
		It follows for the polar coordinate vectors:
		\begin{align}
			\omega^h(\partial_r) &= \omega(\partial_r) = 0 \\
			\omega^h(\partial_\phi) &= \omega(2\partial_\phi) = 1
		\end{align}
		from which it follows that $\omega^h = d\phi$.
		\item This follows directly from 2.:
		\begin{align}
			\theta(\gamma) = \int_{im(\gamma)}\omega^h = \int_{im(\gamma)} d\phi = 2\pi \cdot W(\gamma; 0)
		\end{align}
		\item The smallest winding numbers of smooth curves are 1 for $\Sigma \setminus \{0\}$ (with geometric phase $2\pi$) and $1/2$ for $\Sigma$ (with geometric phase $\pi$), respectively. 
		It follows for the holonomy groups:
		\begin{align}
			Hol(\Sigma \setminus \{0\}) &= \{(k \cdot 2\pi) \text{  mod  } 2\pi : k \in \Z \} = \{0\}  \\
			Hol(\Sigma) &= \{(k \cdot \pi) \text{  mod  } 2\pi : k \in \Z \} = \{0, \pi\} \cong \Z_2
		\end{align}
	\end{enumerate}
\end{proof}

\begin{remark}
	The covering extends naturally to $Sym(2,\R) = L \oplus \Sigma$ by setting:
	\begin{align}
		\pi_{Sym(2,\R)} : \overline{Sym(2,\R) } \coloneqq Sym(2,\R)  &\to Sym(2,\R)  \\
		A = A_L + A_{\Sigma} &\mapsto A_L + \pi_{\Sigma}(A_{\Sigma})
	\end{align}
	Note that the covering space $\overline{Sym(2,\R)}$ has a continuum of branching points, namely the line $L$.
\end{remark}
\begin{remark}
	By lifting to the covering space $\overline{Sym(2,\R)}$, we have achieved that only those closed curves which map eigenvectors back to themselves are considered closed, thus trivializing the holonomy group of $Sym(2,\R) \setminus L$.
	However, the holonomy group of the entire space $Sym(2,\R)$, which includes half-integer winding numbers of curves crossing $L$, got reduced from $Z_4$ to $Z_2$.
	To eliminate it completely, we can lift once again to $\overline{\overline{Sym(2,\R)}}$.
\end{remark}

\section{Application: Mass-spring systems}
Consider a mechanical system with $N$ degrees of freedom whose dynamics is governed by a potential energy function
\begin{align}
	V : \R^N &\to \R \\
	x &\mapsto V(x)
\end{align}
which maps every element $x$ of the position space $\R^N$ to a real number.
We will assume that there is a unique equilibrium position $x_0 \in \R^N$ in which $\nabla V(x_0) = 0$.
Then, using the harmonic approximation of Newton's second law, the dynamics close to $x_0$ is determined up to second order by the Hessian $\nabla^2V(x_0)$, which is a symmetric matrix.

Assume further that the mechanical system in question depends on a number of parameters which can be represented as an elements in a smooth parameter manifold $K$.
Note that the dimension of $K$ (i.e. the number of independent parameters) does not need to agree with the number $N$ of degrees of freedom.
In this set-up, for each element $k \in K$ we obtain a symmetric matrix, namely the Hessian around the equilibrium.
This leads to a mapping:
\begin{align}
	F : K &\to Sym(N, \R) \\
	k &\mapsto F(k) = \nabla^2V(x_0; k)
\end{align}
If we know a metric $g$ on $Sym(N, \R)$ which parallelizes eigenvectors (i.e. one whose Levi-Civita connection is the Berry connection), we can pull it back to a metric $F^*g$ on the parameter space $K$ and do all computations and analysis there.

We present this strategy on mass-spring systems consisting of two mass points ($N = 2$) moving on a straight line with a different number and topology of springs connecting them.

\subsection{Fixed boundary condition}
Consider a mechanical system consisiting of two unit mass points on a straight line, connected to each other and to two fixed walls with three springs.
If the three spring constants are $(\kappa_1, \kappa_2, \kappa_3) \in K \coloneqq \R^3$, the stable length of the springs is $a \in \R$, and the change of length of the $i$-th spring compared to its stable length is $\Delta l_i$, then the energy function is given by:
\begin{align}
	V : \R^2 &\to \R \\
	(x_1, x_2) &\mapsto \frac{1}{2}(\kappa_1 \Delta l_1^2 + \kappa_2 \Delta l_2^2 + \kappa_3 \Delta l_3^2) \\
	&= \frac{1}{2}(\kappa_1 (x_1-a)^2 + \kappa_2(x_2-x_1-a)^2 + \kappa_3(3a-x_2-a)^2)
\end{align}
Its Hessian is:
\begin{align}
\nabla^2V(x_0) = 
\begin{pmatrix}
\kappa_1+\kappa_2 & -\kappa_2 \\ -\kappa_2 & \kappa_2+\kappa_3
\end{pmatrix}
= \Phi^{-1}\left(\frac{1}{2}(\kappa_1 - \kappa_3), -\kappa_2, \frac{1}{2}(\kappa_1 + 2\kappa_2 + \kappa_3)\right).
\end{align}
where the last term is the expression in cartesian coordinates $(x,y,z)$ of $Sym(2, \R)$ (see (\ref{cart-coords})).

The mapping $F$ is then given in coordinates by:
\begin{align}
	\Phi \circ F : K=\R^3 &\to \R^3 \\
	(\kappa_1, \kappa_2, \kappa_3) &\mapsto  \left(\frac{1}{2}(\kappa_1 - \kappa_3), -\kappa_2, \frac{1}{2}(\kappa_1 + 2\kappa_2 + \kappa_3)\right)
\end{align}
which is a linear automorphism of $\R^3$.
We can represent $F$ by the matrix:
\begin{align}
	\mathcal{M} \coloneqq Mat(F) = \begin{pmatrix}
		\frac{1}{2} & 0 & -\frac{1}{2} \\
		0 & -1 & 0 \\
		\frac{1}{2} & 1 & \frac{1}{2}
	\end{pmatrix}.
\end{align}
The pullback metric $F^*g$ on $K$ is thus given by:
\begin{align}
	F^*g = \mathcal{M}^\top g \mathcal{M} = \frac{1}{4r^2}
	\begin{pmatrix}
		r^2+3x^2 & 2r^2 - 6xy & r^2-3x^2 \\
		2r^2-6xy & 8r^2+12y^2 & 2r^2+6xy \\
		r^2 - 3x^2 & 2r^2 + 6xy & r^2+3x^2
	\end{pmatrix}.
\end{align}
where $(x,y,z) = \left(\frac{1}{2}(\kappa_1 - \kappa_3), -\kappa_2, \frac{1}{2}(\kappa_1 + 2\kappa_2 + \kappa_3)\right)$ and $r^2 = x^2+y^2$.

\subsection{Open boundary condition}
Now, consider a mechanical system consisiting of two unit mass points on a straight line, connected only to each other with one spring. If the one spring constant is $\kappa \in K \coloneqq \R$, the stable length of the spring is $a \in \R$, and the change of length of the spring compared to its stable length is $\Delta l$, then the energy function is given by:
\begin{align}
	V : \R^2 &\to \R \\
	(x_1, x_2) &\mapsto \frac{1}{2}\kappa \Delta l^2 \\
	&= \frac{1}{2}\kappa(x_2 - x_1 - a)^2
\end{align}
whose Hessian is:
\begin{align}
	\nabla^2 V(x_0) = \begin{pmatrix} \kappa & -\kappa \\ -\kappa & \kappa \end{pmatrix} = \Phi^{-1}(0, -\kappa, \kappa)
\end{align}
where the last term is the expression in cartesian coordinates $(x,y,z)$ of $Sym(2, \R)$ (see (\ref{cart-coords})).

The mapping $F$ is then given in coordinates by:
\begin{align}
	\Phi \circ F : K=\R &\to \R^3 \\
	\kappa &\mapsto (0, -\kappa, \kappa)
\end{align}
which is a linear embedding of $K = \R$ into $\R^3$.
We can represent $F$ by the matrix:
\begin{align}
	\mathcal{M} \coloneqq Mat(F) = \begin{pmatrix}
		0 \\
		-1\\
		1 
	\end{pmatrix}.
\end{align}
The pullback metric $F^*g$ on $K$ is thus the 1-by-1 matrix given by:
\begin{align}
	F^*g = \mathcal{M}^\top g \mathcal{M} = \left(2+3\frac{y^2}{r^2}\right) \equiv (5).
\end{align}

\subsection{Periodic boundary condiditon}
Finally, consider a mechanical system consisiting of two unit mass points on a flat circle, obtained from a straight line $\R$ by the equivalence relation $x \sim x + 2a$. 
Let the first mass point be connected to the second with one spring and the second mass points connect back to the first one with another
spring. 
If the two spring constants are $(\kappa_1, \kappa_2) \in K := \R^2$, the stable length of the springs is $a \in \R$, and the change of length of the $i$-th spring compared to its stable length is $\Delta l_i$, then the energy function is given by:
\begin{align}
	V : \R^2 &\to \R \\
	(x_1, x_2) &\mapsto \frac{1}{2}(\kappa_1 \Delta l_1^2 + \kappa_2 \Delta l_2^2) \\
	&= \frac{1}{2}(\kappa_1 (x_2-x_1-a)^2 + \kappa_2(x_2+2a-x_1-a)^2)
\end{align}
Its Hessian is:
\begin{align}
\nabla^2V(x_0) = 
\begin{pmatrix}
\kappa_1+\kappa_2 & -\kappa_1-\kappa_2 \\ -\kappa_1-\kappa_2 & \kappa_1+\kappa_2
\end{pmatrix}
= \Phi^{-1}\left(0, -\kappa_1 - \kappa_2, \kappa_1 + \kappa_2 \right).
\end{align}
where the last term is the expression in cartesian coordinates $(x,y,z)$ of $Sym(2, \R)$ (see (\ref{cart-coords})).

The mapping $F$ is then given in coordinates by:
\begin{align}
	\Phi \circ F : K=\R^3 &\to \R^3 \\
	(\kappa_1, \kappa_2, \kappa_3) &\mapsto  \left(0, -\kappa_1 - \kappa_2, \kappa_1 + \kappa_2 \right)
\end{align}
which is a singular linear mapping from $\R^2$ to $\R^3$.
We can represent $F$ by the matrix:
\begin{align}
	\mathcal{M} \coloneqq Mat(F) = \begin{pmatrix}
		0 & 0\\
		-1 & -1 \\
		1 & 1
	\end{pmatrix}.
\end{align}
The pullback metric $F^*g$ on $K$ is thus given by:
\begin{align}
	F^*g = \mathcal{M}^\top g \mathcal{M} = \left(2+3\frac{y^2}{r^2}\right) \cdot 
	\begin{pmatrix}1&1\\1&1\end{pmatrix} \equiv \begin{pmatrix}5&5\\5&5\end{pmatrix} .
\end{align}

\begin{remark}
The metric in the case of periodic boundary condition is singular.
The kernel is spanned by the vector $(1, -1) \in TK$.
This is because the physical system is equivalent to one with open boundary condition with spring constant $\kappa_1 + \kappa_2$.
It is therefore clear that changing the parameter vector $k = (\kappa_1, \kappa_2)$ by a multiple of $(1, -1)$ does not affect the system.
Indeed, in this case the metric measures no change.
\end{remark}

\begin{figure}[h]
\includegraphics[width=\textwidth]{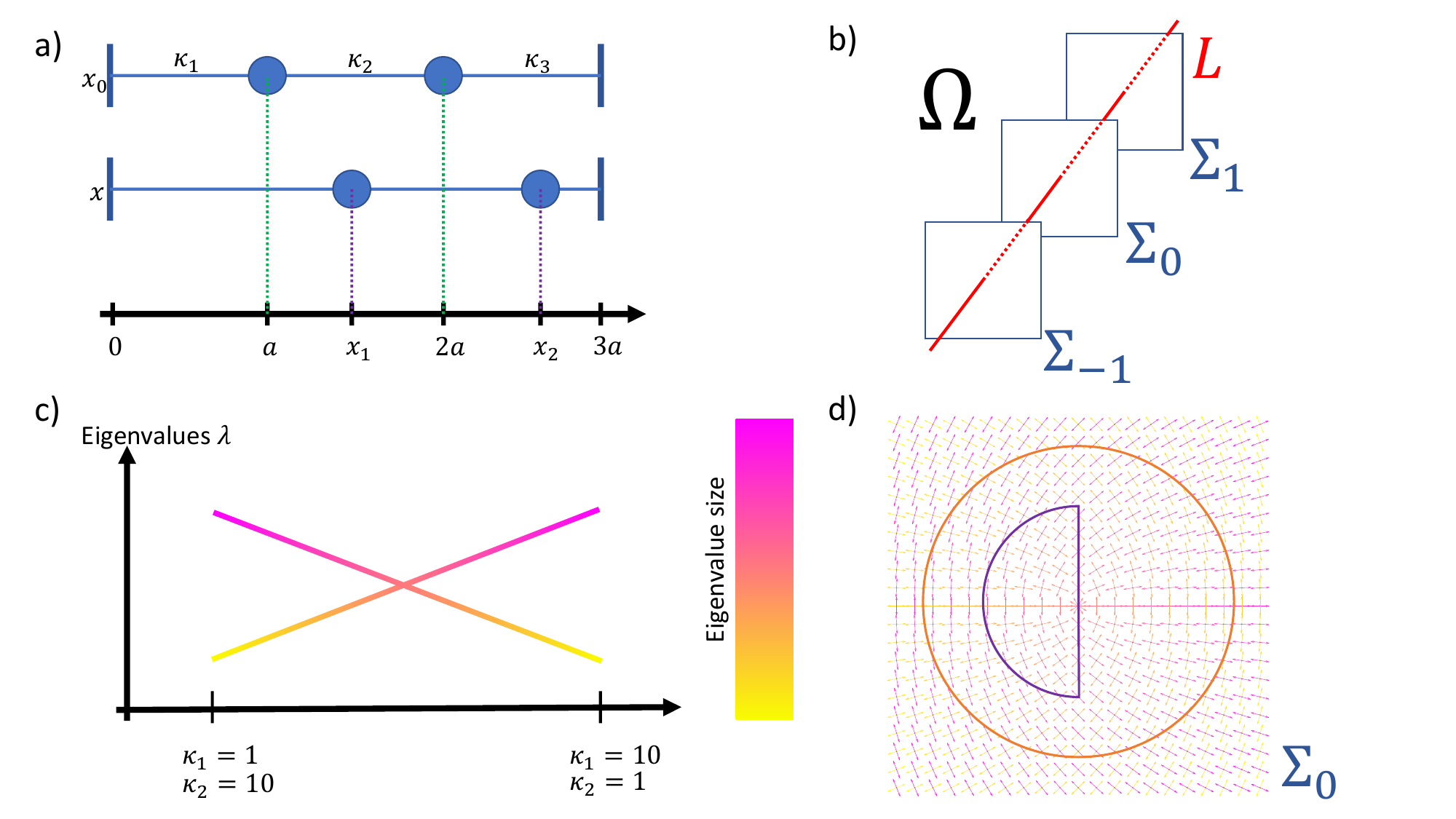}
\caption{a) The spring system in stable state $x_0$ and in a disturbed state $x$. b) The space $\Omega$ of symmetric matrices with its affine-linear subspaces. c) Eigenvalue exchange along a line crossing $L$, e.g. the diameter of the half-circle. d) A slice $\Sigma_0$ with unit eigenvectors plotted at the underlying matrices $A \in \Sigma_0$. The colors of the vectors represent the size of the corresponding eigenvalues. Orange circle as an example of a curve with winding number $\pm 1$; eigenvectors rotate by $\pm \pi$ when moving along it. Purple half-circle through $L$ as an example of a curve with winding number $\pm \frac{1}{2}$; eigenvectors rotate by $\pm \frac{\pi}{2}$ along it.}
\centering
\end{figure}

\section{Funding}
The authors acknowledge the support of the Cluster of Excellence ``Matters of Activity. Image Space Material'' funded by the Deutsche Forschungsgemeinschaft (DFG, German Research Foundation) under Germany's Excellence Strategy – EXC 2025 – 390648296.

\appendix
\section{Intrinsic and extrinsic formulae for Berry connection and curvature} \label{app:conn}
Let $U \subset \R^2$ be an open domain and
\begin{align}
	f : U \to \R^3
\end{align}
an embedding, i.e. an injective, smooth map so that for each $p \in U$ the Jacobian 
\begin{align}
	f_*|_p \coloneqq Df|_p : T_pU \to T_{f(p)}\R^3
\end{align}
has full rank.
The metric $g$ is given for any two vectors $u, v \in T_pU$ as:
\begin{align}
	g(u,v) \coloneqq \langle f_*u, f_*v \rangle
\end{align}
where $\langle .,.\rangle$ is the standard Euclidean metric on $\R^3$.

Let $e_1, e_2 : U \to TU$ be two orthonormal vector fields and let $E_i \coloneqq f_*e_i$ for $i \in \{1, 2\}$. 
Furthermore, let $\nu = E_3 \coloneqq E_1 \times E_2$ be a unit normal field (where $\times$ is the ``cross product'' on $\R^3$). 
Then $(E_1, E_2, \nu)$ is an orthonormal frame in $\R^3$ along $f(U)$. 
Each of those three fields is a function of the two coordinates of $U$ and we may look at the two corresponding partial derivatives. 
As the fields have constant lengths, their derivatives have no component along themselves.
Furhtermore, as they are orthogonal, moving the derivative to the other factor is antisymmetric:
\begin{align}
	0 &= \partial_i \langle E_j, E_j \rangle = 2\langle E_j, \partial_i E_j \rangle \\
	0 &= \partial_i \langle E_j, E_k \rangle = \langle \partial_i E_j, E_k \rangle + \langle E_j, \partial_i E_k \rangle 
\end{align}
In particular, the derivatives of the normal are tangential; the two components define the second fundamental form:
\begin{align}
	II(U, V) \coloneqq -\langle D_U \nu, V \rangle 
\end{align}
for all $U,V \in T\R^3$ tangential to $f(U)$.
We obtain the following formulae:
\begin{align}
	\partial_i E_1 &= -\langle E_1, \partial_i E_2\rangle E_2 + II(\partial_i, E_1) \nu \\
	\partial_i E_2 &= \langle E_1, \partial_i E_2\rangle E_1 + II(\partial_i, E_2) \nu
\end{align}
Projecting onto the tangent space of $f(U)$ we obtain the covariant derivatives:
\begin{align}
	\nabla_i E_1 &= -\langle E_1, \partial_i E_2\rangle E_2 \\
	\nabla_i E_2 &= \langle E_1, \partial_i E_2\rangle E_1
\end{align}
Defining 
\begin{align}\label{con1}
	\omega_i \coloneqq \langle E_1, \partial_i E_2\rangle = \langle E_1, \nabla_i E_2\rangle
\end{align}
we obtain the short-hand form:
\begin{align}
	\nabla_i \begin{pmatrix}
		E_1 \\ E_2
	\end{pmatrix}
= \begin{pmatrix}
	0 & -\omega_i \\
	\omega_i & 0
\end{pmatrix}
\begin{pmatrix}
	E_1 \\ E_2
\end{pmatrix}
\end{align}
We see that, given the tangent orthonormal frame $(E_1, E_2)$, the functions $\omega_i : f(U) \to \R$ determine the covariant derivative $\nabla$. 
They are the components of the connection 1-form $\omega = \omega_1 dx^1 + \omega_2 dx^2$.

We can now move the definition of the covariant derivative to the domain $U$: as the fields $\nabla_i E_j$ are tangent to $f(U)$ and $f_*$ has full rank everywhere, there exist vector fields $\nabla_i e_j$ defined by the property
\begin{align}
	f_*\left(\nabla_i e_j \right)= \nabla_i E_j
\end{align}
This defines the operator $\nabla$ on $U$. 
Note that it can also be computed purely from $g$ and its derivatives without any reference to the normal $\nu$ or the embedding $f$ via: 
\begin{align}
	g(\nabla_k \partial_l, \partial_i) = \frac{1}{2}\sum_m g^{im} \left(\partial_l g_{mk} + \partial_k g_{ml} - \partial_m g_{kl} \right)
\end{align}
It is the Levi-Civita connection of $g$.

It follows for the connection 1-form: 
\begin{align}\label{con2}
	\omega_i = \langle E_1, \nabla_i E_2\rangle = \langle f_*e_1, f_* (\nabla_i e_2)\rangle = g(e_1, \nabla_i e_2)
\end{align}
We now have two formulae for the connection 1-form: an extrinsic formula (\ref{con1}) via the flat ambient space $\R^3$ and an intrinsic formula (\ref{con2}) via the metric $g$ and its Levi-Civita connection.
Differentiating both formuale leads to various expressions for the curvature form $d\omega$.

Extrinsic:
\begin{align}
	d\omega_{ij} &= \partial_i \omega_j  - \partial_j \omega_i \\
		&= \langle \partial_i E_1, \partial_j E_2 \rangle 
			+ \langle E_1, \partial_i \partial_j E_2 \rangle 
			- \langle \partial_j E_1, \partial_i E_2 \rangle 
			- \langle E_1, \partial_j\partial_i E_2 \rangle \\
		&= \langle \partial_i E_1, \partial_j E_2 \rangle - \langle \partial_j E_1, \partial_i E_2 \rangle \\
		&= II(\partial_i, E_1) II(\partial_j, E_2) - II(\partial_j, E_1) II(\partial_i, E_2)
\end{align}
To better recognize the above term note that we can for one fixed point $p \in U$, without loss of generality, choose the coordinates so that the frame $(\partial_1|_{f(p)}, \partial_2|_{f(p)})$ is equal to $(E_1|_{f(p)}, E_2|_{f(p)})$.
In these coordinates, the metric matrix becomes trivial (at this one point $p$) and we get:
\begin{align}
	d\omega_{12} = \det II = \det(g)^{-1}\cdot \det(S) = \det(S) = K
\end{align}
where $S$ is the shape operator and $K$ the Gauss-curvature.
As $(E_1, E_2)$ is an orthonormal frame, it follows that the area element of $f(U)$ is $dA = E_1^\flat \wedge E_2\flat$ and at the point $p$: $dA = \partial_1^\flat \wedge \partial_2^\flat = dx^1 \wedge dx^2$. In total:
\begin{align}
	d\omega = \frac{1}{2} \sum_{i,j=1}^2 d\omega_{ij} dx^i \wedge dx^j = d\omega_{12} dx^1 \wedge dx^2 = K \cdot dA
\end{align}

Intrinsic:

First note that for any $i,j\in\{1,2\}$, $\nabla_i e_1$ is a multiple of $e_2$ and $\nabla_j e_2$ a multiple of $e_1$, so: 
\begin{align}
	g(\nabla_i e_1, \nabla_j e_2) = 0
\end{align}
This is not the case extrinsically, i.e. if one replaces $\nabla$ with $\partial$, $g$ with $\langle.,.\rangle$ etc.
The property that does not hold intrinsically is the commutativity of derivatives.
Indeed, we obtain:
\begin{align}
	d\omega_{ij} &= \partial_i \omega_j  - \partial_j \omega_i \\
	&= g( \nabla_i e_1, \nabla_j e_2 )
	+ g( e_1, \nabla_i \nabla_j e_2 )
	- g( \nabla_j e_1, \nabla_i e_2 )
	- g( e_1, \nabla_j\nabla_i e_2 ) \\
	&= g(e_1, \nabla_i \nabla_j e_2 - \nabla_j \nabla_i e_2) \\
	&= R(\partial_i, \partial_j, e_2, e_1) \\
	&= -R(\partial_i, \partial_j, e_1, e_2)
\end{align}
We can again choose good coordinates at one point $p \in U$ so that $\partial_i|_p = e_i|_p$.
It follows:
\begin{align}
	d\omega_{12} &= R(\partial_1, \partial_2, \partial_2, \partial_1) \\
		&= Ric(\partial_2, \partial_2) \\
		&= \frac{1}{2} Scal
\end{align}
where $Ric$ is the Ricci-curvature (the trace of $R$ along the first and the fourth index) and $Scal$ the scalar curvature (the trace of the Ricci-curvature).
Finally:
\begin{align}
	d\omega = \frac{1}{2}Scal \cdot dA
\end{align}

\section{Eigenvectors and -values of 2-by-2 symmetric matrices} \label{app:eig}
\begin{proposition}
	Let $A \in Sym(2, \R)$ be the 2-by-2 real-valued symmetric matrix with cartesian coordinate $(x,y,z)$ and cylindrical coordinates $(r, \phi, z)$, i.e.
	\begin{align}
		A = \begin{pmatrix} x+z & y \\ y & -x+z \end{pmatrix} = \begin{pmatrix} r\cos(\phi)+z & r\sin(\phi) \\ r\sin(\phi) & -r\cos(\phi)+z \end{pmatrix}
	\end{align}
	Then 
	\begin{align}
		E_1 = \pm \begin{pmatrix} \cos\left(\frac{\phi}{2}\right) \\ \sin\left(\frac{\phi}{2}\right) \end{pmatrix}, \qquad E_2 = \pm \begin{pmatrix} -\sin\left(\frac{\phi}{2}\right) \\ \cos\left(\frac{\phi}{2}\right) \end{pmatrix}
	\end{align}
	are unit-length eigenvectors of $A$ with eigenvalues
	\begin{align}
		\lambda_1 = z+r, \qquad \lambda_2 = z-r,
	\end{align}
    respectively.
\end{proposition}
\begin{proof}
	Compute the change of coordinates from (the candidates for) the eigenbasis $(E_1, E_2)$ to the standard basis:
	\begin{align}
		& \begin{pmatrix} \cos\left(\frac{\phi}{2}\right) & -\sin\left(\frac{\phi}{2}\right) \\ \sin\left(\frac{\phi}{2}\right) & \cos\left(\frac{\phi}{2}\right)\end{pmatrix}
		\begin{pmatrix} z+r & 0 \\ 0 & z-r \end{pmatrix}
		\begin{pmatrix} \cos\left(\frac{\phi}{2}\right) & -\sin\left(\frac{\phi}{2}\right) \\ \sin\left(\frac{\phi}{2}\right) & \cos\left(\frac{\phi}{2}\right)\end{pmatrix}^{-1} \\
		&= \begin{pmatrix} \cos\left(\frac{\phi}{2}\right) & -\sin\left(\frac{\phi}{2}\right) \\ \sin\left(\frac{\phi}{2}\right) & \cos\left(\frac{\phi}{2}\right)\end{pmatrix}
		\begin{pmatrix} z+r & 0 \\ 0 & z-r \end{pmatrix}
		\begin{pmatrix} \cos\left(\frac{\phi}{2}\right) & \sin\left(\frac{\phi}{2}\right) \\ -\sin\left(\frac{\phi}{2}\right) & \cos\left(\frac{\phi}{2}\right)\end{pmatrix} \\
		&= \begin{pmatrix} \cos\left(\frac{\phi}{2}\right) & -\sin\left(\frac{\phi}{2}\right) \\ \sin\left(\frac{\phi}{2}\right) & \cos\left(\frac{\phi}{2}\right)\end{pmatrix}
		\begin{pmatrix} (z+r) \cos\left(\frac{\phi}{2}\right) & (z+r) \sin\left(\frac{\phi}{2}\right) \\ -(z-r)\sin\left(\frac{\phi}{2}\right) & (z-r) \cos\left(\frac{\phi}{2}\right)\end{pmatrix} \\
		&= \begin{pmatrix} (z+r) \cos^2\left(\frac{\phi}{2}\right) + (z-r)\sin^2\left(\frac{\phi}{2}\right) & 2r \sin\left(\frac{\phi}{2}\right) \cos\left(\frac{\phi}{2}\right) \\ 2r \sin\left(\frac{\phi}{2}\right) \cos\left(\frac{\phi}{2}\right) & (z+r) \sin^2\left(\frac{\phi}{2}\right) + (z-r)\cos^2\left(\frac{\phi}{2}\right)\end{pmatrix} \\
		&= \begin{pmatrix} z+r\cos(\phi) & r\sin(\phi) \\ r\sin(\phi) & z -r\cos(\phi) \end{pmatrix}
	\end{align}
\end{proof}

\bibliography{references.bib}
\nocite{*}
\end{document}